\documentclass[11pt]{article}
\usepackage[T1]{fontenc}
\usepackage[ansinew]{inputenc}
\usepackage{amssymb,amsmath,amsthm,eucal,mathrsfs,array,graphicx,xcolor,verbatim}
\usepackage[all]{xy}

\usepackage[hmargin=3cm,vmargin=2.5cm,bindingoffset=0.5cm]{geometry}

\newtheorem{thm}{Theorem}[section]
\newtheorem{lem}[thm]{Lemma}
\newtheorem{rem}[thm]{Remark}
\newtheorem{prop}[thm]{Proposition}

\theoremstyle{definition}

\newtheorem{defn}[thm]{Definition}

\def \a {\alpha}

\def \g {\gamma}
\def \tg {\tilde{\g}}
\def \d {\delta}

\def \e {\varepsilon}

\def \k {\kappa}
\def \l {\lambda}
\def \r {\rho}
\def \s {\sigma}

\def \E {\mathbb{E}}
\def \H {\mathbb{H}}

\def \R {\mathbb{R}}

\def \C {\mathcal{C}}

\def \Cov {\mathrm{Cov}}

\def \SLE {\mathrm{SLE}}

\def \sm {\setminus}

\def \Var {\mathrm{Var}}

\title{Natural parametrization of SLE: the Gaussian free field point of view }
\author{St\'ephane Benoist}

\begin{document}

\maketitle

\begin{abstract}
We provide another construction of the natural parametrization of SLE$_\kappa$ \cite{LawShe_NP,LawRez_MinkNatParam} for $\kappa < 4$. We construct it as the expectation of the quantum time \cite{QZIP}, which is a random measure carried by SLE in an ambient Gaussian free field. This quantum time was built as the push forward on the SLE curve of the Liouville boundary measure, which is a natural field-dependent measure supported on the boundary of the domain. We moreover show that the quantum time can be reconstructed as a chaos on any measure on the trace of SLE with the right Markovian covariance property. This provides another proof that natural parametrization is characterized by its Markovian covariance property.
\end{abstract}

\setcounter{tocdepth}{2}
\tableofcontents

\pagebreak

\begin {center}
{\sc Acknowledgements}
\end {center}

\noindent
Some of the ideas in this paper were developed thanks to a group reading of \cite{QZIP} held during the \emph{Random Geometry} semester at the Newton Institute. I want to thank all the participants of this reading group, and in particular Juhan Aru, Nathana{\"e}l Berestycki, Ewain Gwynne, Nina Holden and Xin Sun for the many fruitful lectures and discussions. I also thank Scott Sheffield for discussions that allowed me to move forward on this project.

\section{Introduction}

The SLE$_\kappa$ are natural random planar curves introduced by Schramm to describe scaling limits of critical interfaces in statistichal mechanics models. These curves have Hausdorff dimension $1+\kappa/8$ (for $0 < \k \leq 8$)  \cite{Bef_Haus} and the corresponding volume measure called natural parametrization has received some attention as the conjectural scaling limit of the length
of discrete interfaces (this is known for $\k=2$ \cite{LaVi_lerwnp}).  

Formally, the natural parametrization was first constructed and characterized as the increasing part of a certain submartingale for SLE \cite{LawShe_NP}, and later shown to be the Minkowski content of SLE \cite{LawRez_MinkNatParam}.

In this paper, we use ideas form Liouville Quantum Gravity (LQG) to provide another description of the natural parametrization for parameters $\kappa < 4$. We also use these methods to reprove a characterization of natural parametrization through its Markovian covariance property. These ideas would also apply to study the natural parametrization of SLE$_{\kappa'}$, for $\kappa'\in(4,8)$, as well as the natural volume measure on boundary touching points of an SLE$_\kappa(\rho)$ in the range of parameters $\kappa \in (0,4)$ and $\rho \in (-2, \kappa/2 -2)$, although we will not discuss this in detail here. Note that the LQG theory corresponding to $\kappa=4$ is critical, and the techniques of this paper do not work for SLE$_4$.

Our study of natural parametrization relies on objects and ideas introduced in \cite{QZIP}. LQG provides random volume measures in a perturbation of the flat metric given by a Gaussian free field, such that classical Euclidean volume measures can be recovered as the expectation of these LQG volume measures. We construct natural parametrization as such an expectation: direct computation shows that the expectation of a certain LQG volume measure, the quantum time, satisfies the axiomatic properties of the natural parametrization. The fact that natural parametrization is characterized by these axiomatic properties  is in turn proved by showing that any LQG measure built on a measure satisfying these axiomatic properties has to be equal to the natural LQG measure on SLE, the quantum time.

In Section \ref{sec:b}, we will provide some background on natural parametrization and on LQG. In Section \ref{sec:charac}, we use the free field to reprove a Markovian characterization of natural parametrization (Theorem \ref{thm:characnp}) by showing a relationship between natural parametrization and quantum time (Proposition \ref{prop:bl}). In Section \ref{sec:construction}, we construct the natural parametrization as an expectation of volume measures coming from the free field (Theorem \ref{thm:isnp}). Finally, we briefly explain in Section \ref{sec:other} how the methods of this paper could be adapted to other setups.

\section{Background}\label{sec:b}

\subsection{Schramm-Loewner evolutions}

Chordal Schramm-Loewner evolutions (SLEs) are a one parameter family of conformally invariant random curves defined in simply-connected domains of the complex plane, with prescribed starting point and endpoint on the boundary.

Let us first give the definition of $\SLE_\kappa$ in the upper half-plane $(\mathbb{H},0,\infty)$. It is a random curve $\eta : \R^+ \rightarrow \overline{\mathbb{H}}$, growing from the boundary point $0$ to $\infty$.

Suppose that such a curve $\eta$ is given to us. Let $H_s$ be the unbounded connected component of $\mathbb{H} \sm \eta([0,s])$, and consider the uniformizing map $g_s : H_s \rightarrow \mathbb{H}$, normalized at $\infty$ such that $g_s(z) = z + 2a_s/z + o(1/z)$. The quantity $a_s$ is the so-called half-plane capacity of the compact hull $K_s=\H\sm H_s$ generated by $\eta([0,s])$. Under additional assumptions\footnote{The curve $\eta$ needs to be instantaneously reflected off its past and the boundary in the following sense: the set of times $s$ larger than some time $s_0$ that $\eta$ spends outside of the domain $H_{s_0}$ should be of empty interior.}, the half-plane capacity $a_s$ is an increasing bijection of $\R^+$, and so we can reparametrize our curve by $t=a_s$.

With this parametrization, the family of functions $g_t$ solves the Loewner differential equation:
\begin{equation}\nonumber
\left\{\begin{array}{l}g_0(z) = z\\
\partial_t g_t(z) = \frac{2}{g_t(z)-W_{t}},\end{array}\right.
\end{equation}
where $W_t= g_t(\eta_t)$ is the (real-valued) driving function.

Conversely, starting from a continuous real-valued driving function, it is always possible to solve the Loewner equation, and hence to recover a family of compact sets $K_t$ in $\overline{\H}$, growing from $0$ to $\infty$, namely $K_t$ is the set of initial conditions $z$ that yield a solution $g_u(z$) blowing up before time $t$. It may happen that the compact sets $K_t$ coincides with the set of hulls generated by the trace of a curve $\eta$, which can in this case be recovered as $\eta_t= \lim_{\e\rightarrow 0} g_t^{-1}(W_t +i \e)$.
\begin{defn}\label{prop:SLE}
The process $\SLE_\kappa$ in ($\H,0,\infty$) is the curve obtained from the solution of the Loewner equation with driving function $W_t = \sqrt{\kappa} B_t$, where $B_t$ is a standard Brownian motion.
\end{defn}

The law of $\SLE_\kappa$ in ($\H,0,\infty$) is invariant by scaling. Hence, given a simply-connected domain $(D,a,b)$ with two marked points on its boundary, we can define the $\SLE_\kappa$ in  $(D,a,b)$ to be the image of an $\SLE_\kappa$ in  ($\H,0,\infty$) by any conformal bijection $(\H,0,\infty) \rightarrow (D,a,b)$.

We now restrict to values of the parameter $\kappa <4$. The SLE curves almost surely are simple curves of dimension $d=1+\frac{\kappa}{8}$ \cite{Bef_Haus}, and they carry a non-degenerate volume measure of dimension $d$.

\begin{defn}[\cite{LawRez_MinkNatParam}]
The $d$-dimensional Minkowski content $\mu$ of the trace of $\SLE_\k$ is a non-trivial measure. We call it the \emph{natural parametrization} of {\rm SLE}.
\end{defn}

\begin{rem}Given a conformal isomorphism $\phi:D\rightarrow\phi(D)$, the natural parametrization transforms via the natural covariance formula for $d$-dimensional measures: $\mu_{\phi(D)}\circ\phi=|\phi'|^{d} \mu_D$.
\end{rem}

SLE curves have the following spatial Markov property:
\begin{prop}\label{prop:markov-sle} Let $\eta$ be an $\SLE_\kappa$ in $ ($D,a,b$)$ and $\tau$ an almost surely finite stopping time. Then the law of the future of the curve $(\eta(u))_{u\geq \tau}$ conditioned on its past $(\eta(u))_{0\leq u \leq \tau}$ is that of an $\SLE_\kappa$ in $(D\setminus \eta([0,\tau]),\eta_\tau,b)$.
\end{prop}
As the natural parametrization is a local deterministic function of SLE, the same Markov property holds for SLEs together with their natural parametrization.

We will state in Section \ref{sec:charnp} that this property actually characterizes the natural parametrization.

\subsection{The Gaussian free field}

Let us recall general facts about Gaussians, before defining the Gaussian free field.

\subsubsection{Gaussians}

Gaussians are usually associated to vector spaces carrying a non-degenerate scalar product. However, it is natural to extend this definition in the degenerate case, by saying that a centered Gaussian of variance $0$ is deterministically $0$.
\begin{defn}
The Gaussian on a vector space $V$ equipped with a symmetric positive semi-definite bilinear form $(\cdot,\cdot)$ is the joint data, for every vector $v\in V$, of a centered Gaussian random variable $\Gamma_v$, such that $v\mapsto\Gamma_v$ is linear, and such that for any couple $v,w\in V$, the covariance $\text{Cov}(\Gamma_v,\Gamma_w)=(v,w)$.
\end{defn}
Heuristically, one should think of $\Gamma_v$ as being the scalar product $(h,v)$ of $v$ with a random vector $h$ drawn according to the Gaussian law $e^{-\frac{1}{2}(h,h)}dh$. However, the linear form $v\rightarrow \Gamma_v$ is (in infinite-dimensional examples) a.s. not continuous, and so there does not exist a vector $h\in V$ such that $\Gamma_v=(h,v)$ for all $v\in V$. One can nonetheless try to find such a random object $h$ in a superspace of $V$: this is the question of finding a continuous version of Brownian motion, or of seeing the Gaussian free field as a distribution.

\subsubsection{Definition of the Gaussian free field}
We now fix a smooth and bounded simply-connected Jordan domain $D$ of the plane. Let us consider the degenerate Dirichlet scalar product
$$
(f,g)_\nabla=\frac{1}{2\pi}\int_D{\nabla f\nabla g}
$$
on the space $\C^{\infty}_\nabla\left(\overline{D}\right)$ of continuous functions $f$ on $\overline{D}$ that are smooth on $D$ and such that the Dirichlet norm $||f||_\nabla:=(f,f)_\nabla^{1/2}$ is finite.

Let $\C^{\infty,0}_\nabla\left(\overline{D}\right)$ be the subspace of $\C^{\infty}_\nabla\left(\overline{D}\right)$ that consists of functions vanishing on the boundary $\partial D$. The Dirichlet scalar product is non-degenerate on this subspace, and we denote by $H^0(D)$ its completion.

We define $H(D)$ to be the completion of the whole space $\C^{\infty}_\nabla\left(\overline{D}\right)$ with respect to the (non-degenerate) metric $||f||_\nabla+|f(x_0)|$, where $x_0\in D$ is an arbitrary point. The space $H(D)$ is naturally endowed with the (degenerate) Dirichlet scalar product

\begin{defn}
The Gaussian free field on $D$ with Dirichlet (resp. Neumann) boundary conditions is the Gaussian on the space $H^0(D)$ (resp $H(D)$) equipped with the (renormalized) Dirichlet scalar product $(\cdot,\cdot)_\nabla$.
\end{defn}
The free field is the joint data, for any function $f\in H^0(D)$ (resp $f\in H(D)$), of a random variable $\Gamma_f$.
The Dirichlet product being conformally invariant, so is the Gaussian free field with either boundary conditions. Hence, both these Gaussian free fields can be defined in any simply-connected domain of the complex plane.

\subsubsection{The Gaussian free field as a random distribution}

We would like to see the Gaussian free field as a random distribution, i.e. we want to be able to write
\begin{eqnarray}\label{eq:gff}
\Gamma_f=(h,f)_\nabla,
\end{eqnarray}
where $h$ is a random distribution. Hence, we are looking for a way that is consistent with (\ref{eq:gff}) to define a coupling of the quantities
\begin{eqnarray}
(h,g)=\int_D hg \nonumber
\end{eqnarray} for any smooth and compactly supported function $g$. Let $\Delta=\partial_x^2+\partial_y^2$ be the Laplacian, and $\partial_n$ denotes the outward normal derivative on the boundary. If $H$ is a regular enough distribution, and $f$ is a smooth enough function, then the Green formula holds:
\begin{equation}\label{eq:Green}
2\pi(H,f)_\nabla=-(H,\Delta f)+\int_{\partial D} H \partial_n f.
\end{equation}

Suppose we can solve the Poisson problem with Dirichlet boundary conditions for a certain smooth function $g$, i.e. suppose we can find a function $f\in H^0(D)$ such that 
\begin{equation}\nonumber
\Delta f =g\ {\rm on}\ D
\end{equation}

Then, for $h$ a Dirichlet free field (recall that, at least formally, $h=0$ on $\partial D$), we can tentatively define $(h,g)$ in agreement with Green's formula by
$$
(h,g):=-{2\pi}\Gamma_f.
$$

Similarly, suppose we can solve the Poisson problem with Neumann boundary conditions for a certain smooth function $g$, i.e. suppose we can find a function $f\in H(D)$ such that 
\begin{equation}\nonumber
\left\{
\begin{array}{l}\Delta f =g\ {\rm on}\ D\\
\partial_n f=0\ {\rm on}\ \partial D\nonumber.\end{array}\right.
\end{equation}
Then, for $h$ a Neumann free field, we can tentatively define $(h,g)$ in agreement with Green's formula by
$$
(h,g):=-{2\pi}\Gamma_f.
$$

We now recall a classical result of analysis.
\begin{prop}\label{prop:Poisson}
For any function $g$, the Poisson problem with Dirichlet boundary conditions admits a unique solution.

The Poisson problem with Neumann boundary conditions admits solutions if and only if the function $g$ satisfies the integral condition $\int_D g = 0$.
\end{prop}

In terms of finding a distributional representation $h$ of the Gaussian free field $\Gamma$, Proposition \ref{prop:Poisson} implies that the Dirichlet free field $h$ is canonically defined in the space of distributions. However, in the Neumann case, we only have a natural way to define the pairing $(h,g)$ for test functions $g$ such that $\int_D g =0$. In other words, the distribution $h$ representing the Neumann free field is canonically defined only in the space of distributions modulo constants.

We can make some choice to fix the constant for $h$. This is done by picking a function $\rho$ of non-zero mean, and by deciding that $(h,\rho)$ should have a certain joint distribution with the set of random variables $(h,g)$ where $g$ runs over all test functions satisfying the integral constraint $\int_D g =0$. However, none of these choices will preserve the conformal invariance of the field.

In the following we will always assume that some choice of constant for the Neumann free field has been made.

\subsubsection{Covariance}

\begin{defn}
The pointwise covariance of the field is the generalized function $K(x,y)$ that represents the bilinear form $(g,\tilde{g})\mapsto\E[(h,g)(h,\tilde{g})]$:
$$
\int_{z,w\in D} g(z)K(z,w)\tilde{g}(w):=\E[(h,g)(h,\tilde{g})].
$$
\end{defn}
In $\H$, we have that $K(z,w)=-\log|z-w|+\log|z-\overline{w}|$ for the Dirichlet free field, and for a Neumann field normalized such that $(h,\rho)=0$, we have that
$$
K(z,w)= G(z,w) - \int \rho(x)G(x,w){\rm d}x - \int G(z,y)\rho(y){\rm d}y + \int\int \rho(x)G(x,y)\rho(y){\rm d}x{\rm d}y,
$$
where $G(z,w)=-\log|z-w|-\log|z-\overline{w}|$.

\subsubsection{Markov property for the field}

Let us now state a spatial Markov property for the Dirichlet free field on a domain $D$.

Consider $K$ a hull of $D$, i.e. a subset of $D$ such that $D'=D\sm K$ is a simply-connected open domain. Note that a function $\C^{\infty,0}_\nabla\left(\overline{D'}\right)$ can be extended by $0$ on $K$ into a function belonging to $H^0(D)$. Hence the space of functions $H^0(D)$ splits as an orthogonal sum $H^{0}(D')\oplus F$. This translates into the following relationship between the Dirichlet Gaussian free fields $h$ and $h'$ on $D$ and $D'$.
\begin{prop}\label{prop:Markovfield}
The Gaussian free field $h$ splits as an independent sum $h=h'+C$, where the correction $C$ is a Gaussian field, of covariance $\Cov(h)-\Cov(h')$.
\end{prop}

Similarly, we have the following relationship between a Neumann and Dirichlet free field on a domain $D$.
\begin{prop}\label{prop:DN}
The Neumann free field on $D$ can be written as the independent sum of a Dirichlet free field on $D$ and a random harmonic function. 
\end{prop}

\begin{proof}
This follows from the fact that the orthogonal complement of $H^0(D)$ in $H(D)$ is the space of harmonic functions on $D$ (this can be seen via the Green formula (\ref{eq:Green}), .
\end{proof}

\subsection{Volume measures in a free field}

\subsubsection{Chaos measures}\label{sec:chaos}

From a field $h$ and a reference measure $\sigma$, one can build interesting random measures $e^{\tg h}\s$ called \emph{chaos measures}. The field $h$ is usually too irregular for its exponential to make sense, and so defining chaos requires some renormalization.

Let $\sigma$ be a Radon measure whose support has Hausdorff dimension at least $d$, and let $h=\tilde{h}+m$, where $m$ is a continuous function and $\tilde{h}$ is a Gaussian field with logarithmic correlations, i.e. such that there is some positive real number $n$ such that, for points $x$ belonging to a set of full measure for $\sigma$, the covariance $K$ of the field satisfies
$$
K(z,z+\d)=-n\log|\d| + O_{|\d|}(1).
$$
We can then build non-trivial \emph{chaos measures} $:\! e^{\tilde{\gamma}h}\s\! :$ for values of the parameter $0<\tg<\sqrt{2d/n}$ (see \cite{Ber_Chaos} and references therein).

In this paper, we will exclusively construct chaos on measures supported in the bulk of $\H$ for a field with $n=1$ logarithmic correlations or on measures supported on $\R$ for a field with $n=2$ logarithmic correlations (such as the Gaussian free field on $\H$ with Neumann boundary conditions). Let us define chaos measures in these two cases.

We fix a radially-symmetric smooth function $\theta_i^1$ of total mass $1$ supported in the unit ball around $i\in\H$, and let the bulk regularizing function be given by
$$
\theta_z^\e(w)=\e^{-2}\theta_i^1\left(\frac{w-z}{\e}+i\right).
$$
For $z\in\R$, we let the boundary regularizing function be
$$
\tilde{\theta}_x^\e(w)=2\theta_z^\e(w).
$$

\begin{defn}\label{def:chaos}
The $\tilde{\gamma}$-chaos of a Radon measure $\sigma$ with respect to a field $h$ with logarithmic correlations is defined in the following way.

For a measure $\s$ supported in the bulk $\H$, $n=1$ and $\tg<\sqrt{2d}$:
$$
:\! e^{\tilde{\gamma}h}\s\! :({\rm d}z) := \lim_{\e \rightarrow 0} e^{\tilde{\g}(\theta_z^\e,h)} \e^{\frac{\tilde{\g}^2}{2}} \s({\rm d}z),
$$
and for a measure $\sigma$ supported on the boundary $\R$, $n=2$, and $\tg<\sqrt{d}$:
$$
:\! e^{\tilde{\gamma}h}\s\! :({\rm d}x) := \lim_{\e \rightarrow 0} e^{\tilde{\g}(\tilde{\theta}_x^\e,h)} \e^{\tilde{\g}^2} \s({\rm d}x).
$$
\end{defn}

We can recover the reference measure $\sigma$ from its chaos either by reversing the procedure used to construct chaos, or by averaging on the field.

\begin{prop}\label{prop:recovc}
For a measure $\s$ supported on the bulk,
$$
\lim_{\e \rightarrow 0} e^{-\tilde{\g}(\theta_z^\e,h)} \e^{-\frac{\tilde{\g}^2}{2}}:\! e^{\tilde{\gamma}h}\s\! :({\rm d}z) :=  \s({\rm d}z).
$$
\end{prop}

\begin{prop}\label{prop:recove}
Let $h$ be a Gaussian free field, and $\s$ a Radon measure supported in the bulk. We can recover the measure $\s$ from its chaos $:\! e^{\tilde{\gamma}h}\s\! :$, as
$$
\s(dz)=e^{-\frac{\tilde{\g}^2}{2} \widehat{K}(z)}\E\left[:\! e^{\tilde{\gamma}h}\s\! :(dz)\right],
$$
where $\widehat{K}(z) = \lim_{\e \rightarrow 0} \Var(\theta_z^\e,h) +\log\e$.

\end{prop}
\begin{proof}
The first equality holds by uniform integrability (see \cite[Section 2.3]{Ber_Chaos}).
\begin{eqnarray}
\E\left[:\! e^{\tilde{\gamma}h}\s\! :(dz)\right]&=&\lim_{\e \rightarrow 0}\E\left[e^{\tilde{\g}(\theta_z^\e,h)} \e^{\frac{\tilde{\g}^2}{2}} \s(dz)\right]\nonumber\\
&=&\lim_{\e \rightarrow 0} \E\left[e^{\tilde{\g}(\theta_z^\e,h)}\right] \e^{\frac{\tilde{\g}^2}{2}} \s(dz)\nonumber\\
&=&\lim_{\e \rightarrow 0} e^{\frac{\tilde{\g}^2}{2}\Var(\theta_z^\e,h)} \e^{\frac{\tilde{\g}^2}{2}} \s(dz) \nonumber\\
&=&\lim_{\e \rightarrow 0} e^{\frac{\tilde{\g}^2}{2} (-\log\e + \widehat{K}(z) +o_\e(1))} \e^{\frac{\tilde{\g}^2}{2}} \s(dz) \nonumber\\
&=& e^{\frac{\tilde{\g}^2}{2} \widehat{K}(z)} \s(dz). \nonumber
\end{eqnarray}
\end{proof}

\subsubsection{Liouville quantum gravity}
 
The goal of Liouville quantum gravity (LQG) is the study of \emph{quantum surfaces}, i.e. of complex domains $D$ carrying a natural random metric - the Liouville metric. This Liouville metric is of the form $e^{\gamma h}g$ where $g$ is some metric compatible with the complex structure, and $h$ is a field related to the Gaussian free field on $D$. It is actually unclear how to properly define this random metric (except in the case $\gamma=\sqrt{8/3}$, see \cite{MiSh_LQGTBMI}). We can however build as chaos measures certain Hausdorff volume measures of the Liouville metric.

Let us stress that the object of interest is the Liouville metric, and its canonical volume measures. The field $h$ only appears as a tool to construct these objects in fixed coordinates, and hence should change appropriately when we change coordinates, so that the geometric objects are left unchanged. 
\begin{defn}The Liouville coordinate change formula is given by:
$$
h_{\phi(D)}=h_D\circ \phi^{-1} + Q\log |{\phi^{-1}}'|,
$$
where $Q=\frac{\g}{2}+\frac{2}{\g}$. 
\end{defn}
Natural volume measures are then invariant under this change of coordinates (see Proposition \ref{prop:b}).

We call \emph{quantum surface} a class of field-carrying complex domains $(D,h)$ modulo Liouville changes of coordinates. A particular representative $(D,h)$ of a given quantum surface is called a \emph{parametrization}.

\subsubsection{The radial part of the free field}

Let us consider a Neumann free field $h$ in the upper half-plane $\H$. Note that the Dirichlet space $H(\H)$ admits an orthogonal decomposition $H^r \oplus H^m$ for the Dirichlet product, where $H^r$ is the closure of radially symmetric functions $f(|\cdot|)$, and $H^m$ is the closure of functions that are of mean zero on every half-circle $C_R=\{z\in\H, |z|=R\}$. The field $h$ correspondingly splits as a sum $h^r+h^m$. If one fixes the constant of the field by requiring that $h^r(1)=0$, the components $h^r$ and $h^m$ are independent.

The law of the radial component $h^r(e^{-\frac{t}{2}})$ can be explicitly computed: it is a double-sided Brownian motion (see the related \cite[Proposition 3.3]{DuSh_LQGKPZ}), i.e. the functions $\left(h^r(e^{-\frac{t}{2}})\right)_{t\geq 0}$ and $\left(h^r(e^{\frac{t}{2}})\right)_{t\geq 0}$ are independent standard Brownian motions.
Moreover, note that adding a drift $a t$ to $h^r(e^{-t})$ corresponds to adding the function $-a \log|\cdot|$ to the field.

\subsubsection{The wedge field}

We now define, in the upper half plane $\H$, an object closely related to the Neumann free field: the \emph{wedge field}.
Let us fix some real number $\a<Q=\frac{\g}{2}+\frac{2}{\g}$.
\begin{defn}\label{def:A}
The $\a$-wedge field is a random distribution that splits in $H^r \oplus H^m$ as an independent sum $h_W=h^r_W+h^m$, where $h^m$ is as for the Neumann free field, and $h^r_W(e^{-t})$ has the law of $A_t$, as defined below.
\end{defn}
For $t>0$, $A_t=B_{2t}+\a t$, where $B$ is a standard Brownian motion started from $0$.
For $t<0$, $A_t=\widehat{B}_{-2t}+\a t$, where $\widehat{B}$ is a standard Brownian motion started from $0$ independent of $B$ and conditioned on the singular event
$$
A_t-Qt=\widehat{B}_{-2t}+ (\a-Q) t > 0
$$
for all negative times t.

\begin{rem}[{\cite[Proposition 4.6.(i)]{MatingTrees}}]\label{rem:invwed} 
The law of an $\alpha$-wedge field is invariant as a quantum surface (i.e. up to a Liouville change of coordinates) when one adds a deterministic constant to the wedge field.
\end{rem}

\begin{rem}[{\cite[Proposition 4.6.(ii)]{MatingTrees}}]
The wedge field in $\H$ appears as a scaling limit of a Neumann free field when zooming in at the origin. Moreover, wedge fields are absolutely continuous with respect to the Gaussian free field on compact subsets of $\overline{\H}\setminus\{0\}$. In particular, this local absolute continuity allows one to build chaos measures for wedge fields. 
\end{rem}

\subsection{The unzipping operation}

We will be cutting surfaces open along curves drawn on them.

Let $(\eta(t))_{t\geq 0}$ be a simple curve growing in $\H$ from $0$ to $\infty$ (for example, $\eta$ is an $\SLE_\k$, with $\k < 4$). For a positive time $t$, let $\phi^t_0:\H\sm \eta([0,t])\rightarrow\H$ be the uniformizing map normalized so that $\phi^t_0(z) = z+O_{z\to\infty}(1)$ and $\phi^t_0(\eta^0(t))=0$, and let $\phi_t^0$ be its inverse.

The zipping and unzipping operators are the family of conformal maps parametrized by $s,t\geq 0$ such that $\phi_s^t=\phi_0^t\circ\phi_s^0$.

For any time $t\geq 0$, the curve unzipped by $t$ units of time is given by
$$
\left(\eta^t(u)\right)_{u\geq 0}:=\left(\phi^t_0(\eta(t+u))\right)_{u\geq 0}.
$$

If the curve $\eta$ comes with a volume measure $\mu$ of dimension $d$, the corresponding volume measure on the unzipped curve $\eta^t$ is given by 
\begin{eqnarray}\label{eq:unzmes}
\mu^t:= |{\phi_t^0}'|^{-d} \ \mu \circ\phi_t^0
\end{eqnarray}

Finally, if $(\H,h,\eta)$ is a quantum surface with a simple curve drawn on it (see Figure \ref{fig:zipper}), the unzipped field $h^t$ at time $t$ is given by the Liouville change of coordinates
$$
h^t:=h\circ \phi^0_t + Q\log |{\phi^0_t}'|.
$$

\begin{figure}
\centering
\includegraphics[width=12cm]{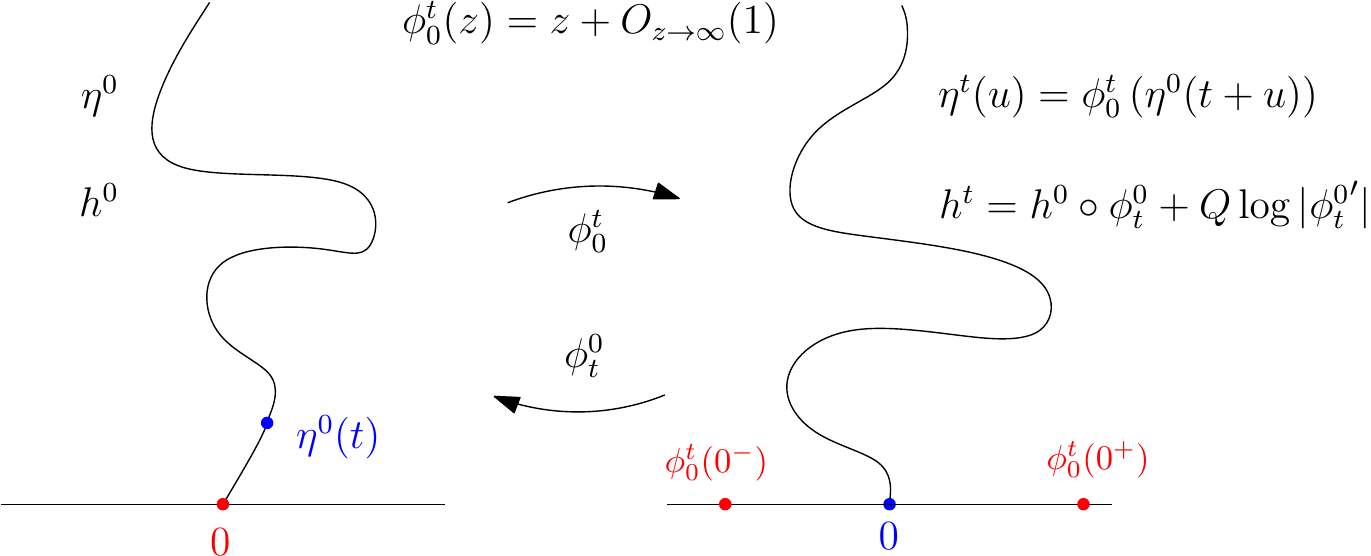}
\caption{Zipping up and unzipping a quantum surface.}
\label{fig:zipper}
\end{figure}

\subsection{Volume measures and unzipping}

\subsubsection{Characterization of natural parametrization.}\label{sec:charnp}

Let $\eta$ be an $\SLE_\kappa$, for $\kappa < 4$. We treat measures on $\eta$ as volume measures of dimension $d=1+\kappa/8$, in the sense that we use (\ref{eq:unzmes}) to push them through the unzipping operation.

\begin{thm}[\cite{LawShe_NP}]\label{thm:characnp}
Let $\mu$ be a locally finite measure on $\eta((0,\infty))$. 

Suppose that the following holds: for any time $t>0$, conditionnaly on $\eta([0,t])$, the couple $(\eta^t, \mu^t)$ has the law of $(\eta,\mu)$.

Then the measure $\mu$ is uniquely determined up to a deterministic multiplicative constant.
\end{thm}

This result appears slightly differently in \cite{LawShe_NP}, and for more general values of $\kappa$. We give a new proof of this statement in Section \ref{sec:charac}.

\subsubsection{Natural Liouville measures are invariant under changes of coordinates.}

 We now tune $\kappa$ and $\gamma$ so that $\gamma^2=\kappa<4$. Let $\lambda$ be the Lebesgue measure on $\R$, and let $\eta$ be an SLE$_\kappa$, that comes with a $d$-dimensional volume measure $\mu$. Let $h$ be an independent Gaussian field with appropriate logarithmic correlations (recall Section \ref{sec:chaos}).

\begin{prop}\label{prop:b}
The Liouville boundary measure $:\! e^{\frac{\gamma}{2}h}\l\! :$ as well as the chaos on natural parametrization $:\! e^{\frac{\gamma}{2}h}\mu\! :$ are invariant by Liouville changes of coordinates. In particular, the following holds for times $0\leq s<t$ (see Figure \ref{fig:measures}):
\begin{eqnarray}
:\! e^{\frac{\gamma}{2}h^t}\l:_{|\phi^t_s(\R)}\circ \phi^t_s\ =\ :\! e^{\frac{\gamma}{2}h^s}\l\! :\label{eq:lebisinvariant}\\
:\! e^{\frac{\gamma}{2}h^t}\mu^t\! :\ = \ :\! e^{\frac{\gamma}{2}h^s}\mu^s\! :_{|\phi^s_t(\eta^t)}\circ \phi^s_t. \label{eq:npisinvariant}
\end{eqnarray}
\end{prop}

The observation (\ref{eq:npisinvariant}), even though straightforward to check, is one of the main novelties of this paper. It will allow us to give a new characterization of natural parametrization by repeating an argument of \cite{QZIP} (see Section \ref{sec:charac}).

\begin{proof}Let us first prove (\ref{eq:npisinvariant}). Recall that the dimension of $\SLE$ is given by $d=1+\frac{\g^2}{8}$, and that the Liouville change of coordinates in the context of the unzipping process reads $h^t=h^s\circ \phi^s_t + Q\log |{\phi^s_t}'|$.  We have that:
\begin{eqnarray}
:\! e^{\frac{\gamma}{2}h^s}\mu^s\! :_{|\phi^s_t(\eta^t)}\circ \phi^s_t &=& \lim_{\e \rightarrow 0} e^{\frac{\g}{2}(\theta_{\phi^s_t(\cdot)}^\e,h^s)} \e^{\frac{\g^2}{8}} d\mu^s\circ \phi^s_t \nonumber\\
&=& \lim_{\e \rightarrow 0} \e^{\frac{\g^2}{8}} e^{\frac{\g}{2} (\theta_{\cdot}^{\e{\phi^t_s}'},h^s\circ \phi^s_t)} |{\phi^s_t}'|^d d\mu^t \nonumber\\
&=& \lim_{\d \rightarrow 0} (|{\phi^s_t}'|\d)^{\frac{\g^2}{8}} e^{\frac{\g}{2} (\theta_{\cdot}^{\d},h^t) - \frac{\g}{2} Q\log |{\phi^s_t}'|}|{\phi^s_t}'|^d d\mu^t  \nonumber\\
&=& \lim_{\d \rightarrow 0} |{\phi^s_t}'|^{\frac{\g^2}{8}-\frac{\g}{2} Q+d} \d^{\frac{\g^2}{8}} e^{\frac{\g}{2} (\theta_{\cdot}^{\d},h^t)} d\mu^t \nonumber\\
&=& |{\phi^s_t}'|^{\frac{\g^2}{8}-\frac{\g}{2} (\frac{\g}{2}+\frac{2}{\g})+ 1 +\frac{\g^2}{8}} :\! e^{\frac{\gamma}{2}h^t}\mu^t\! : \nonumber \\
&=& :\! e^{\frac{\gamma}{2}h^t}\mu^t\! :. \nonumber
\end{eqnarray}
Note that to go from the first to the second line, we used that
$$
(\theta_{\phi^s_t(z)}^\e,h^s)-(\theta_{z}^{\e|{\phi^t_s}'\circ\phi_t^s(z)|},h^s\circ \phi^s_t)
$$
is a Gaussian of variance $o(1)$ as $\e$ goes to $0$. This is seen in the following way: with
$$
F_\e(\cdot)=\theta_{\phi^s_t(z)}^\e(\cdot)-|{\phi^t_s}'(\cdot)|^2\theta_{z}^{\e|{\phi^t_s}'\circ\phi_t^s(z)|}(\phi^t_s(\cdot)),
$$
and $K$ being the covariance of the field, one has that
$$
\int_{\H^2} F_\e(w) K(w,y) F_\e(y){\rm d}w{\rm d}y= o_\e(1).
$$

A similar computation for the boundary Liouville measure yields
\begin{eqnarray}
:\! e^{\frac{\gamma}{2}h^t}\l\! :_{|\phi^t_s(\R)}\circ \phi^t_s = |{\phi^t_s}'|^{\frac{\g^2}{4}-\frac{\gamma}{2}Q+1} :\! e^{\frac{\gamma}{2}h^s}\l\! : = :\! e^{\frac{\gamma}{2}h^s}\l\! :. \nonumber
\end{eqnarray}
Invariance under general changes of coordinates follows from the same computations.
\end{proof}

\section{The Markovian characterization of natural parametrization}\label{sec:charac}

We now provide a new proof of the fact that the natural parametrization of SLE is characterized by its Markovian property (Theorem \ref{thm:characnp}).  The core idea (Proposition \ref{prop:bl}) is to show that the $\gamma/2$-chaos on any measure on SLE that has the correct Markovian covariance property has to be Sheffield's quantum time, i.e. the push-forward by the zipping up operation of the Liouville boundary measure.

From now on, we work in the upper half-plane $(\H,0,\infty)$, and we tune LQG and SLE parameters so that $\gamma^2=\k<4$. Moreover, we fix the law of a measure $\mu$ coupled with an SLE$_\kappa$ $\eta$ such that $\mu$ is a measure on the trace of $\eta$ that satisfies the Markovian property of Theorem \ref{thm:characnp} (we will refer to $\mu$ as a \emph{Markovian volume measure}).

Let us first introduce the quantum zipper, which we will use in this proof.

\subsection{The quantum zipper}

The \emph{quantum zipper} $(h^t_W,\eta^t)_{t\geq 0}$ is a process of pairs consisting of a distribution and a curve, that is obtained by unzipping the initial conditions: a $(\g-\frac{2}{\g})$-wedge field $h_W$, and an independent $\SLE_\k$ curve $\eta$. We choose the time parametrization of $\eta$ to correspond to the chaos $:\! e^{\frac{\gamma}{2}h_W}\mu^0\! :$, i.e. we have 
$$
:\! e^{\frac{\gamma}{2}h_W}\mu\! :\left(\eta([0,t])\right)=t.
$$

\begin{prop}[{\cite[Theorem 1.8]{QZIP}}]\label{prop:quantstat}
The quantum zipper $(h^t_W,\eta^t)_{t\geq 0}$ has stationary law, when the fields are considered up to Liouville changes of coordinates.
\end{prop}
\begin{rem}
Note that the time-parametrization of the quantum zipper is invariant by Liouville changes of coordinates (Proposition \ref{prop:b}).
\end{rem}

\subsection{The quantum time is a chaos on any Markovian volume measure}\label{sec:QTasC}

\begin{figure}[ht]
\centering
\includegraphics[width=12cm]{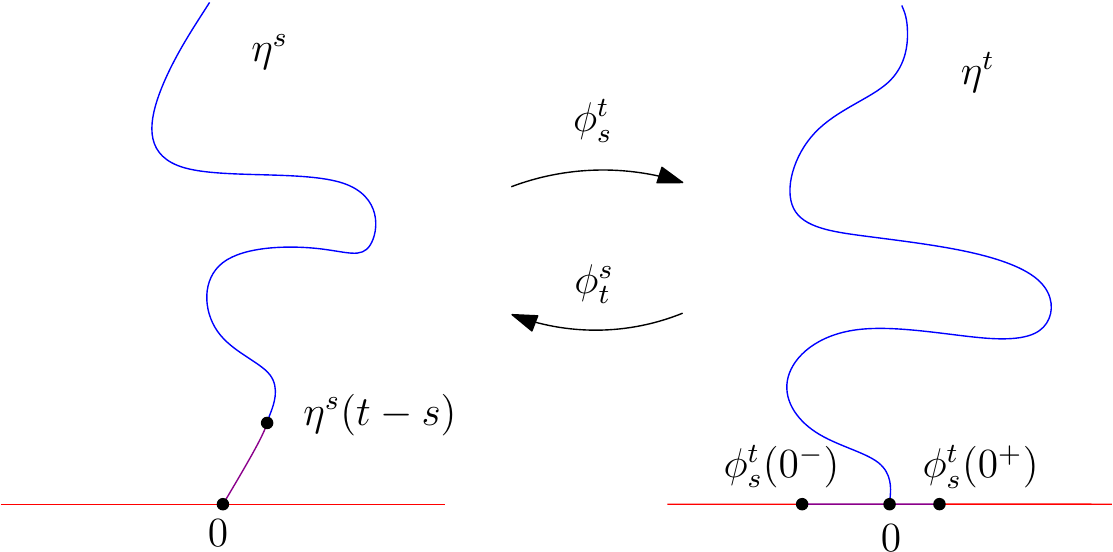}
\caption{Under the unzipping operation, the natural volume measures on the curve $\eta$ and on the boundary $\R$ are preserved: Proposition \ref{prop:b} (i) and (ii) respectively claim that the red (resp. blue) regions carry the same measures. On the purple regions, the natural measures on the curve and on the boundary coincide (Proposition \ref{prop:bl}).}
\label{fig:measures}
\end{figure}

\begin{prop}\label{prop:bl}
Let us consider the quantum zipper $(h^t_W,\eta^t)_{t\geq 0}$, and let $\mu$ be a Markovian volume measure on $\eta$. Then, for any times $0\leq s<t$:
$$
:\! e^{\frac{\gamma}{2}h_W^t}\l\! :_{|[0,\phi^t_s(0^+)]}\circ \phi^t_s = C:\! e^{\frac{\gamma}{2}h_W^s}\mu^s\! :_{|\eta^s[0,t-s]},
$$
where $C$ is a constant.
\end{prop}
The constant $C$ depends on the somewhat arbitrary renormalization procedure used to build chaos measures, as well as on the choice of Markovian volume measure $\mu$.

The proof of Proposition \ref{prop:bl} mimicks an argument in the proof of \cite[Theorem 1.8]{QZIP}.
\begin{proof}
Let $m(t)=:\! e^{\frac{\gamma}{2}h_W^t}\l\! :([0,\phi^t_0(0^+)])$ be the Liouville boundary mass of the part of the right-hand side $\SLE$ path that has been unzipped between times $0$ and $t$. Note that for any time $t\geq 0$, 
$$
m(t+1)=m(t)\ +:\! e^{\frac{\gamma}{2}h_W^{t+1}}\l\! :([0,\phi^{t+1}_t(0^+)]).
$$
On the other hand, by stationarity of the quantum zipper up to Liouville change of coordinates (Proposition \ref{prop:quantstat}), and by invariance of volume measures under such changes (Proposition \ref{prop:b}), the quantity $:\! e^{\frac{\gamma}{2}h_W^{t+1}}\l\! :([0,\phi^{t+1}_{t}(0^+)])$ has stationary law. By the Birkhoff ergodic theorem, the quantity $\frac{m(n)}{n}$ almost surely converges towards a random variable $C(\omega)$, for integer times $n$ going to $\infty$. The function $m(t)$ being monotone, this implies that $\frac{m(t)}{t}$ converges almost surely (and hence in probability) towards $C(\omega)$ as the time $t$ goes to $\infty$.

Let us spell out this last statement: there is a random variable $C(\omega)$ such that for any $\e>0$, we can find a deterministic time $T$, such that with probability at least $1-\e$, we have that 
$$
\sup_{t\geq T}\left|\frac{m(t)}{t} - C(\omega)\right| < \e.
$$
Let us now add a constant $\frac{2}{\g}\log\frac{\tau}{T}$ to the field $h_W^0$, where $\tau>0$ is an arbitrary small time. The law of the quantum zipper is preserved (Remark \ref{rem:invwed}), but the time scale $t$ and the quantity $m(t)$ are both scaled by $\frac{\tau}{T}$. Hence, for any $\e>0$, for any time $\tau>0$, with probability at least $1-\e$: 
$$
\sup_{t\geq \tau}\left|\frac{m(t)}{t} - C(\omega)\right| < \e.
$$
In other words, almost surely, $m(t)=C(\omega)t$ for all positive times $t$. The random constant $C(\omega)$ is then measurable with respect to the curve $\eta^0$ and the field $h^0_W$ in any neighborhood of $0$. However, the corresponding $\sigma$-algebra is trivial, and the random variable $C(\omega)$ is hence a deterministic constant.
\end{proof}

\subsection{Proof of the characterization theorem}

\begin{proof}[Proof of Theorem \ref{thm:characnp}]
Proposition \ref{prop:bl} together with Proposition \ref{prop:recovc} tells us that we can recover any Markovian volume measure $\mu$ of $\SLE$ from a wedge field. Indeed, for any positive time $T$,
\begin{eqnarray}
\mu_{|\eta^0[0,T]}&=&\lim_{\e \rightarrow 0} e^{-\frac{\g}{2}(\theta_z^\e,h^0_W)} \e^{-\frac{\g^2}{8}} :\! e^{\frac{\gamma}{2}h^0_W}\mu^0\! :_{|\eta^0[0,T]} \nonumber\\
&=& C\ \lim_{\e \rightarrow 0} e^{-\frac{\g}{2}(\theta_z^\e,h^0_W)} \e^{-\frac{\g^2}{8}}    :\! e^{\frac{\gamma}{2}h_W^T}\l\! :_{|[0,\phi^T_0(0^+)]}\circ \phi^T_0 . \nonumber
\end{eqnarray}
The last expression depends on a choice of a particular Markovian volume measure $\mu$ only through the constant $C$. In particular, all Markovian volume measures are scalar multiples of each other.
\end{proof}

\section{Existence of the natural parametrization}\label{sec:construction}

\subsection{Construction of the natural parametrization}

We now work in the upper half-plane $(\H,0,\infty)$.
Let $h^0$ be a Dirichlet Gaussian free field independent of an SLE$_\kappa$ $\eta^0$ going from $0$ to $\infty$, parametrized by capacity. Let $t>0$ be a positive time. We will explain why the chaos $:\! e^{\frac{\gamma}{2}h^t}\l\!:$ is well-defined in Proposition \ref{prop:pfD}.

\begin{defn}\label{def:np}
We consider the measure $\mu^0$ on SLE given by
$$
\mu^0_{|\eta^0[0,t]} =F(z)\ \E\left[\left.:\! e^{\frac{\gamma}{2}h^t}\l\! :\right|\eta^0\right]\circ \phi_0^t,
$$
where $\lambda$ is the Lebesgue measure on $\R$, and $F(z)= e^{-\frac{\gamma^2}{8}\hat{K}(z)}$, where $\hat{K}$ is as in Proposition \ref{prop:recove} for the field $h^0$.
\end{defn}
 Note that this definition is consistent for different values of $t$ (Proposition \ref{prop:b}). It gives a locally finite measure on $\mu^0$ (see Lemma \ref{lem:finite}).

\begin{thm}\label{thm:isnp}
The measure $\mu^0$ is (up to constant) the natural parametrization of $\SLE$.
\end{thm}

\begin{proof}
Let $s>0$ be a positive time. We will first provide, given the curve $\eta^0$, a construction of the independent Dirichlet free field $h^0$. Let us consider a Dirichlet Gaussian free field $\tilde{h}^s$ independent of $\eta^0$ (this implicitly defines a zipped up field $\tilde{h}=\tilde{h}^0$). Note that the field $\tilde{h}^s\circ \phi_0^s$ is a Dirichlet Gaussian free field on the domain $\H\setminus \eta^0([0,s])$.
Hence, by Proposition \ref{prop:Markovfield}, we can define a correction field $C_0$ which, conditionally on $\eta^0$, is a Gaussian field independent of $\tilde{h}^s$ and of covariance
$$
\Cov(h^0) - \Cov(\tilde{h}^s\circ\phi_0^s).
$$
In particular, we can compute the variance of $C_0$ at a point $z$ to be
$$
\hat{K}(z) - \hat{K }(\phi_0^s(z)) + \log|{\phi_0^s}'|.
$$

Then, conditionally on $\eta^0$, the field $h^0=\tilde{h}^s\circ \phi_0^s+C_0$ is a Dirichlet Gaussian free field in $\H$. In other words, the field $h^0$ is a Dirichlet Gaussian free field in $\H$, independent of the curve $\eta^0$.

We then define a measure on the curve $\eta^s$ by
$$
\tilde{\mu}^s_{|\eta^s[0,t-s]}=F(z)\ \E\left[\left.:\! e^{\frac{\gamma}{2}\tilde{h}^t}\l\! :\right| \eta^s\right]\circ \phi_s^t.
$$
Note that the couples $(\eta^s,\tilde{\mu}^s)$ and $(\eta^0,\mu^0)$ have same law. Moreover, we will see in Lemma \ref{lem:covar} that $\tilde{\mu}^s=\mu^s$.
Hence, we can conclude thanks to the characterization of natural parametrization (Theorem \ref{thm:characnp}).
\end{proof}

\begin{lem}\label{lem:covar}
We have that $\tilde{\mu}^s=\mu^s$, i.e.
$$
\mu^0_{|\eta^0[s,\infty)}\circ \phi_s^0 = |{\phi_s^0}'|^{1+\frac{\kappa}{8}} \tilde{\mu}^s.
$$
\end{lem}

\begin{proof}
Note that for any time $t>s$, we have
$$
h^t=h^0\circ \phi_t^0 + Q \log|{\phi_t^0}'|=\tilde{h}^t+C_0\circ \phi_t^0 + Q\log|{\phi_s^0}'|\circ \phi_t^s.
$$

Hence we can rewrite
\begin{eqnarray}
\mu^0_{|\eta^0[0,t]}&=&F(z)\, \E\left[\left.\lim_{\e\to 0} e^{\frac{\gamma}{2}(\theta_x^\e,h^t)}\e^{\frac{\gamma^2}{4}}d\lambda \right| \eta^0\right]\circ \phi_0^t\nonumber\\
&=&F(z)\,  \E\left[\left.e^{\frac{\gamma}{2} C_0\circ \phi_t^0} |{\phi_s^0}'\circ\phi_t^s|^{\frac{\gamma}{2}Q} \lim_{\e\to 0} e^{\frac{\gamma}{2}(\theta_x^\e,\tilde{h}^t)}  \e^{\frac{\gamma^2}{4}}d\lambda \right| \eta^0\right]\circ \phi_0^t\nonumber\\
&=&F(z)\, \E\left[\left.e^{\frac{\gamma}{2}C_0\circ \phi_t^0}\right| \eta^0\right]\circ \phi_0^t \cdot |{\phi_s^0}'\circ \phi_0^s|^{\frac{\gamma}{2}Q} \cdot \E\left[\left.:\! e^{\frac{\gamma}{2}\tilde{h}^t}\l\! :_{|[0,\phi^t_0(0^+)]} \right| \eta^0\right]\circ \phi_0^t .\nonumber
\end{eqnarray}

For a centered Gaussian $\mathcal{N}$ of variance $\sigma^2$, $\E\left[e^{\lambda \mathcal{N}}\right]=e^{\frac{\lambda^2 \sigma^2}{2}}$. Hence the contribution of the correction field $C_0$ can be rewritten
\begin{eqnarray}
\E\left[\left.e^{\frac{\gamma}{2}C_0\circ \phi_t^0} \right| \eta^0\right]\circ \phi_0^t &=& \E\left[\left.e^{\frac{\gamma}{2}C_0} \right| \eta^0\right]\nonumber\\
&=& \exp\left(\frac{\gamma^2}{8}\left(\hat{K}(z) - \hat{K}(\phi_0^s(z)) + \log|{\phi_0^s}'|\right)\right)\nonumber\\
&=& \frac{|{\phi_0^s}'|^{\frac{\gamma^2}{8}}F(\phi_0^s(z))}{F(z)}.\nonumber
\end{eqnarray}

Hence, we have for any times $t>s>0$:
\begin{eqnarray}
&&\mu^0_{|\eta^0[s,t]}\circ \phi_s^0 \nonumber\\
&=& F(\phi_s^0 (z))\cdot\frac{|{\phi_0^s}'|^{\frac{\gamma^2}{8}}F(\phi_0^s(z))}{F(z)}\circ \phi_s^0 \cdot |{\phi_s^0}'|^{\frac{\gamma}{2}Q} \cdot \E\left[\left.:\! e^{\frac{\gamma}{2}\tilde{h}^t}\l\! :_{|[0,\phi^t_s(0^+)]} \right| \eta^0\right]\circ \phi_s^t \nonumber\\
&=& F(\phi_s^0 (z))\frac{|{\phi_0^s}'\circ\phi_s^0|^{\frac{\gamma^2}{8}}F(z)}{F(\phi_s^0(z))} \cdot |{\phi_s^0}'|^{1+\frac{\gamma^2}{4}} \cdot \E\left[\left.:\! e^{\frac{\gamma}{2}\tilde{h}^t}\l\! :_{|[0,\phi^t_s(0^+)]}\right| \eta^s\right]\circ \phi_s^t \nonumber\\
&=& |{\phi_s^0}'|^{1+\frac{\gamma^2}{8}} \cdot  F(z) \, \E\left[\left.:\! e^{\frac{\gamma}{2}\tilde{h}^t}\l\! :_{|[0,\phi^t_s(0^+)]} \right| \eta^s\right]\circ \phi_s^t\nonumber\\
&=& |{\phi_s^0}'|^{1+\frac{\kappa}{8}} \tilde{\mu}^s_{|\eta^s[0,t-s]}.\nonumber
\end{eqnarray}
The change of conditionning in the second equality is possible as $\tilde{h}^t$ can be obtained from $\tilde{h}^s$, an object independent of $\eta^0$, by an operation that involves the curve $\eta^0$ only via the map $\phi_s^t$, which is also a function of $\eta^s$.

\end{proof}

\subsection{The push-forward of the Dirichlet free field on the boundary}

Let $h^0$ be a Dirichlet Gaussian free field, and $\eta^0$ an independent SLE$_\kappa$.

\begin{prop}\label{prop:pfD}
The chaos $:\! e^{\frac{\gamma}{2}h^t}\l\! :$ is well-defined.
\end{prop}

\begin{proof}
From the Dirichlet free field $h^0$, we can construct a field $\tilde{h}^0=h^0+g+\frac{2}{\gamma} \log|\cdot|$ which has the law of a Neumann free field plus a singularity $\frac{2}{\gamma} \log|\cdot|$, and where $g$ is a random function, which is harmonic in $\H$ (see Proposition \ref{prop:DN}). By \cite[Theorem 1.2]{QZIP}, the field $\tilde{h}^t$ is also a Neumann free field plus a singularity $\frac{2}{\gamma} \log|\cdot|$, and so the chaos $:e^{\frac{\gamma}{2}\tilde{h}^t}\l\! :$ is well-defined. On the other hand, we see that on the open interval $(\phi_0^t(0^-),\phi_0^t(0^+))$, we can write $h^t=\tilde{h}^t+m$, where $m=g(\phi_t^0)+\frac{2}{\gamma} \log|\phi_t^0|$ is a continuous function. It follows that, on this interval, the chaos $:\! e^{\frac{\gamma}{2}h^t}\l\! :$ is well-defined and equal to $e^{\frac{\gamma}{2}m}:\! e^{\frac{\gamma}{2}\tilde{h}^t}\l\! :$, whereas on its complement,  $:\! e^{\frac{\gamma}{2}h^t}\l\! :$ is naturally the trivial zero measure.
\end{proof}

\subsection{The expectation of quantum time is finite}\label{sec:finite}

We work in the setup of Definition \ref{def:np}.

\begin{lem}\label{lem:finite}
The measure $\mu^0$ is locally finite, i.e. for any positive times $s<t$,
$$
\mu^0\left(\eta^0\left([s,t]\right)\right)<\infty
$$
almsot surely.
\end{lem}

\begin{proof}
We consider the first exit time $\tau$ of the ball of radius $R$ centered at $0$ by the SLE $\eta^0$. There exists a deterministic capacity time $t$ such that $t>\tau$ almost surely, and there similarly exists a deterministic interval of the real line $[-M,M]$ such that $\phi_0^t (\eta^0([0,\tau])) \subset [-M,M]$ almost surely. Let us call $\mathcal{I}^\e$ the set of points of the interval $[-M,M]$ such that $|\phi_t^0(z)|\leq \e$. Showing for any $\e$ that $\E\left[:\! e^{\frac{\gamma}{2}h^t}\l\! : ([-M,M]\setminus\mathcal{I}^\e)\right]<\infty$ will imply that the measure $\mu^0$ gives finite mass to the points of SLE that are outside of the ball of radius $\e$ around the origin and that are located before the first exit of the ball of radius $R$ around the origin. This in particular imply that $\mu^0$ is a locally finite measure.

As in the proof of Proposition \ref{prop:pfD}, we construct (thanks to Proposition \ref{prop:DN}) from the Dirichlet free field $h^0$ a field $\tilde{h}^0=h^0+g+\frac{2}{\gamma} \log|\cdot|$ which has the law of a Neumann free field plus a singularity $\frac{2}{\gamma} \log|\cdot|$. To fix the constant of the Neumann field $\tilde{h}^0$, we consider $\rho$ the indicator function of a disk of unit area centered at $(\sqrt{t}+1) i$. This ensures that the support of $\rho\circ\phi_0^t$ is included in a deterministic bounded region $\mathcal{R}$ of $\mathbb{H}$, and that the function $|{\phi_t^0}'|$ is deterministically bounded and bounded away from $0$ on the region $\mathcal{R}$ and on the support of $\rho$.

We now fix the constant of the field $\tilde{h}^0$ so that $\int \tilde{h}^0 \rho = 0$. By \cite[Theorem 1.2]{QZIP}, we see that $\tilde{h}^t-c$ has the law of $\tilde{h}^0$, where $c$ is the following random constant:

\begin{eqnarray}\label{eq:c}
c = \int \tilde{h}^t \rho = \int \tilde{h}^0  (\rho\circ\phi_0^t)  |{\phi_0^t}'|^2 - Q \int \log |{\phi_t^0}'| \ \rho   
\end{eqnarray}

Note that the first term of the right hand side is dominated by the absolute value of a multiple of a Gaussian of bounded variance, and the second term is deterministically bounded.

We can now bound
\begin{eqnarray}
\E\left[:\! e^{\frac{\gamma}{2}h^t}\l\! : ([-M,M]\setminus\mathcal{I}^\epsilon)\right] &\leq& \E\left[:\! e^{\frac{\gamma}{2}(\tilde{h}^t-\frac{2}{\gamma} \log|\phi_t^0(\cdot)|)}\l\! : ([-M,M]\setminus\mathcal{I}^\e)\right] \nonumber\\
&\leq& \e^{-2/\g}\,  \E\left[:\! e^{\frac{\gamma}{2}\tilde{h}^t}\l\! : ([-M,M]\setminus\mathcal{I}^\e)\right] \nonumber\\
&\leq& \e^{-2/\g}\,  \E\left[e^{\frac{\gamma}{2}c}:\! e^{\frac{\gamma}{2}(\tilde{h}^t-c)}\l\! : ([-M,M])\right]  < \infty. \nonumber
\end{eqnarray}

The first inequality follows from Jensen's inequality: conditionally on $\eta^0$ and $h^0$, the law of $g$ is the same as the law of $-g$. In particular, $\E\left[ \left. e^{\frac{\gamma}{2}g\circ\phi_t^0}\right|\eta^0,h^0 \right]\geq 1 $.

On the last line, we use H\"older inequality. Indeed, the random variable $e^{\frac{\gamma}{2}c}$ can be compared to the exponential of a Gaussian (recall (\ref{eq:c})) and so has finite moments of all order. And $:\! e^{\frac{\gamma}{2}(\tilde{h}^t-c)}\l\! : ([-M,M])$ has the law of $:e^{\frac{\gamma}{2}\tilde{h}^0}\l\! : ([-M,M])$, which has finite moments of order $m$ larger than $1$: its moments are finite for all $0<m<4/\gamma^2$ (\cite[Theorem 2.11]{RhVa_review}).

\end{proof}

\section{Other volume measures on SLE curves}\label{sec:other}

One could similarly study the natural parametrization of SLE$_{\kappa'}$ for $\kappa'\in (4,8)$, as well as the volume measure on boundary touching points of a SLE$_\kappa$($\rho$) for the range of parameters $\kappa \in (0,4)$ and $\rho \in (-2, \kappa/2 -2)$.

In both these cases, quantum times can be defined as the Poisson clock of the process of disks disconnected from $\infty$ by the curve (see \cite[Definition 7.2]{MatingTrees} and \cite[Definition 7.14]{MatingTrees} respectively).

In the first case, the quantum time on SLE$_{\kappa'}$ should correspond to the $2/\gamma$-chaos on the natural parametrization of SLE (where $\gamma=4/\sqrt{\kappa'}$ through the duality relationship $\kappa=16/\kappa'$). In the second case, the quantum time on boundary touching points of an SLE$_\kappa$($\rho$) (which is a set of almost sure Hausdorff dimension $\frac{(\k-4-2\rho)(\rho+4)}{2\k}$ \cite{MiWu_DimSLE}) should be a $\hat{\gamma}$-chaos on the Minkowski content, where $\hat{\gamma}=\frac{\gamma}{2}(1-\frac{2\r+4}{\gamma^2})$, and $\gamma=\sqrt{\kappa}$.

The correct parameters for the chaos measures are determined by the invariance of these measures under change of coordinates, as in Proposition \ref{prop:b}.

\bibliography{biblio}{}
\bibliographystyle{plain}

\end{document}